\documentclass[a4paper,10pt]{article}

\usepackage{amsmath, amsthm, amssymb, hyperref, enumerate, comment}
\usepackage{graphicx}
\usepackage{epstopdf}
\newtheorem{theorem}{Theorem}[section]
\newtheorem{lemma}[theorem]{Lemma}

\newtheorem{corollary}[theorem]{Corollary}

\theoremstyle{remark}

\title{Equal Entries in Totally Positive Matrices\footnote{This work was supported in part by National Science foundation grant \#DMS-0751964}}

\author{
Miriam Farber\thanks{Department of Mathematics, Technion--Israel Institute of Technology, Haifa, IL-32000, Israel, (miriamf@tx.technion.ac.il)}
\and
Mitchell Faulk\thanks{Department of Mathematics, University of Notre Dame, Notre Dame, Indiana 46556, USA, (mfaulk1@nd.edu).}
\and
Charles R. Johnson\thanks{Department of Mathematics, College of William and Mary, Williamsburg, Virginia 23187, USA, (crjohn@wm.edu)}
\and
Evan Marzion\thanks{Department of Mathematics, University of Wisconsin-Madison, Madison, Wisconsin 53706, USA, (marzion@wisc.edu)}}

\begin{document}

\maketitle

\begin{abstract}
We show that the maximal number of equal entries in a totally positive (resp. totally nonsingular) $n\textrm{-by-}n$ matrix is $\Theta(n^{4/3})$ (resp. $\Theta(n^{3/2}$)). Relationships with point-line incidences in the plane, Bruhat order of permutations, and $TP$ completability are also presented. We also examine the number and positionings of equal $2\textrm{-by-}2$ minors in a $2\textrm{-by-}n$ $TP$ matrix, and give a relationship between the location of equal $2\textrm{-by-}2$ minors and outerplanar graphs.
\end{abstract}

\section{Introduction}
An $m\textrm{-by-}n$ matrix is called totally positive (non-negative), $TP$ ($TN$), if every minor of it is positive (non-negative). Such matrices play an important role in many facets of mathematics \cite{Fallat, Gasca}. It is known that $TN$ is the closure of $TP$, that is, there exists an arbitrarily small perturbation of any given $TN$ matrix that is $TP$. However, proofs of this typically perturb all, or most, entries, which is usually not necessary. Here, we advance the question of which are the minimal sets of entries that must be perturbed in a given $TN$ matrix that it become $TP$. Our interest stems, in part, from the likely value of this information in solving two outstanding problems: (1) the Garloff conjecture \cite{Garloff} and (2) the relationship between $TN$ completable and $TP$ completable patterns. In the latter case, author Johnson conjectures that the $TN$ completable patterns are properly contained in the $TP$ completable ones. However, the question of "small" perturbing sets is likely of interest for other reasons, as well as on its own. We are somewhat lax, at times, in the notion of perturbation that we use, however, in that we do not always require our perturbations to be small. For this we use the word "change" in place of "perturb".

Here, we begin the study of minimally perturbing collections by asking which sets of entries in the $J$ matrix (all 1's) need be perturbed.  Because of positive diagonal scaling, $J$ may as well be any positive rank 1 matrix. A matrix is called $TP_k$ ($TN_k$) if every $k\textrm{-by-}k$ sub-matrix of it is $TP$ ($TN$). A positive rank 1 matrix is $TN$ and $TP_1$. A virtue of considering $J$ is due to a very important fact: for any $TP_2$ matrix, some Hadamard power is $TP$ \cite{Fallat}. So, if $J$ can be perturbed to be $TP_2$, it can, using the same entries, be changed to $TP$. Hadamard powering preserves the 1's.

We call a matrix \emph{totally nonsingular}, $TNS$, if \emph{all} its minors are non-zero. To perturb a matrix with all non-zero entries to a $TNS$ matrix, at least one entry of every $2\textrm{-by-}2$ submatrix must be available for perturbation. For this, it is equivalent to consider 0,1 matrices in which no $2\textrm{-by-}2$ sub-matrix of 1's occurs (think of the 1's as unperturbed entries and the 0's as entries that may be perturbed). A certain amount of information about how many 1's such a matrix may have is available \cite{Furedi, Guy, Mendelsohn, Roman}, and it is known, asymptotically, that this is $O(n^{3/2})$ in the $n\textrm{-by-}n$ case \cite{Furedi}. We show that for a matrix with nonzero entries, perturbation of the entries corresponding to the 0's in any $(0,1)$- matrix with no $2\textrm{-by-}2$ matrix of 1's, is sufficient to achieve a $TNS$ matrix. This has the consequence that the maximum number of equal entries in an $n\textrm{-by-}n$ $TNS$ matrix is $\Theta(n^{3/2})$.

We say that a $(0,1)$- matrix is ($TP$-)\emph{changeable} if there is a change of the entries of $J$, corresponding to the 0's, that is $TP$. In the case of changing $J_n$ to $TP$, the greatest number of 1's that can be achieved asymptotically is $O(n^{4/3})$. This means that the maximum number of equal entries in either a $TP_2$ or $TP$ matrix is $\Theta(n^{4/3})$.

There are important, and previously unnoticed, connections between these ideas and a number of other mathematical concepts. As well as equal entries and minimally perturbable sets of entries, these include: incidences of points and lines in the Euclidian plane; the notion of positive orthogonal cycles defined in \cite{Pach}; Bruhat order on permutations; the $TP$ matrix completion problem; and the relative position of points above and below an individual line in a collection.

In the next section we present several useful definitions and known results. We also develop some tools to deal with our questions about total positivity. These tools come primarily from combinatorial geometry, but, surprisingly, the connection with total positivity was not previously noticed.\\

In section 3 we deal with the totally nonsingular case, and show that asymptotically at most $\Theta(n^{3/2})$ entries may be equal in an $n\textrm{-by-}n$ $TNS$ matrix.\\

In section 4 we deal with the connections between changeability and points and lines in the plane. If a configuration is the incidence matrix of points and lines in the plane (properly ordered), then we show that $J$ may be changed to $TP$ via this configuration. Using this, and prior work about incidences, we show that asymptotically at most $\Theta(n^{4/3})$ entries may be equal in an $n\textrm{-by-}n$ $TP$ matrix. In view of this count matching the maximum number of incidences of $n$ points and $n$ lines, we raise the question of whether a collection of equal entries in a $TP$ matrix must correspond to the incidence matrix of some points and lines in the plane.\\

In section 5 we consider and bound the maximum number of entries (and their positions) of an $n\textrm{-by-}n$ $TP$ matrix that lie among the $k$ smallest (among the $k$ largest) entries. (It follows that the smallest or largest cannot attain the high frequency of $\Theta(n^{4/3})$ mentioned above.) We also obtain some further geometric information about points and lines in the plane. The latter involves the frequency of point-line pairs in which the point lies bellow (above) the line (vertically).\\

In section, 6, we apply the ideas developed to $TP$ completion, and in section 7, we show that the notion of "positive orthogonal cycle", introduced in \cite{Pach}, may be characterized by the Bruhat comparability of the two constituent permutations. We also characterize the notion of "positive orthogonal collection" by Bruhat comparability of certain classes of matrices.
(The absence of a collection of) positive orthogonal collection is used to identify $(0,1)$- matrices via which $J$ may be changed to $TP$.\\

Finally, in section 8, we deal with the number and positionings of equal $2\textrm{-by-}2$ minors in a  $2\textrm{-by-}n$ $TP$ matrix. In the main result in this section, we present a relation between outerplanar graphs and the positioning of equal $2\textrm{-by-}2$ minors.
\section{Definitions and known results}
Let $M_{m,n}(\mathbb{F})$ denote the set of $m\textrm{-by-}n$ matrices over a field $\mathbb{F}$. Let $\mathbb{F}_2 = \{0,1\}$ denote the field with two elements. For each $x \in \mathbb{F}$, we define a map $C_x : M_{m,n}(\mathbb{F}) \to M_{m,n}(\mathbb{F}_2)$ by $C_x(A) = [c_{ij}]$ with
\begin{align*}
c_{ij} = \left\{\begin{array}{ll}
1, & a_{ij} = x \\
0, & a_{ij} \ne x
\end{array}
\right.
\end{align*}
if $A = [a_{ij}]$.  We say that $C_x(A)$ is the \emph{configuration of $x$ in $A$}.

The definitions of orthogonal cycle and positive orthogonal cycle are taken from \cite{Pach}. Let $A$ be an $m\textrm{-by-}n$ matrix and let $C = (p_0,p_1\ldots, p_{2k})$ be a sequence of positions of $A$.  We call $C$ an \emph{orthogonal cycle} if $p_0 = p_{2k}$ and for each $i$ satisfying $0 \le i <k$, positions $p_{2i}$ and $p_{2i+1}$ belong to the same row, while positions $p_{2i+1}$ and $p_{2i+2}$ belong to the same column.  If for each $i$ satisfying $0\le i \le2k$ the entry of $A$ in position $p_i$ is one, then $C$ is said to be an \emph{orthogonal cycle of $A$}.

Let $A$ be an $m\textrm{-by-}n$ matrix and let $C = (p_0,\ldots, p_{2k})$ be an orthogonal cycle of $A$.  Let $P(i,j)$ be the set of positions $\{(u,v): u>i, v>j\}$ of $A$.  Define a function
\begin{equation} \label{C}
\mathfrak{C}(i,j) := \\
\left|\{l \in \mathbb{Z}: 0< l \le k, p_{2l} \in P(i,j)\} \right| - \left|\{l\in \mathbb{Z}: 0< l \le k, p_{2l-1} \in P(i,j)\}\right|,
\end{equation}
We call an orthogonal cycle $C$ \emph{positive} if $\mathfrak{C}(i,j) \ge 0$ for each pair $(i,j)$ and $\mathfrak{C}(i,j) > 0$ for at least one such pair.

For example, the ones in the matrix $\left(
  \begin{array}{cccc}
    0 & 1 & 0 & 1 \\
    0 & 0 & 1 & 1 \\
    1 & 1 & 0 & 0 \\
    1 & 0 & 1 & 0 \\
  \end{array}
\right)$ form a positive orthogonal cycle.
In order to calculate the values of $\mathfrak{C}(i,j)$ for the matrix above, consider the matrix
$\left(
  \begin{array}{cccc}
    0 & p_0 & 0 & p_1 \\
    0 & 0 & p_3 & p_2 \\
    p_6 & p_7 & 0 & 0 \\
    p_5 & 0 & p_4 & 0 \\
  \end{array}
\right).$
One can see that $\mathfrak{C}(1,1)=0$, since $|\{p_4,p_2\}|-|\{p_7,p_3\}|=0$, and $\mathfrak{C}(2,2)=|\{p_4\}|=1$. Similarly, $\mathfrak{C}(i,j) \geq 0$ for all the pairs $(i,j)$, and hence the cycle is positive orthogonal.

Let $\mathcal{O} = \{C_a\}_{a \in A}$ be a collection of orthogonal cycles and let $\{\mathfrak{C}_a\}_{a \in A}$ be the associated functions as in (\ref{C}).  We say (similarly to \cite{Pach}) that the collection $\mathcal{O}$  of orthogonal cycles is \emph{positive} if for each $(i,j)$
\begin{align}\label{sum1}
\sum_{a \in A} \mathfrak{C}_a(i,j) \ge 0
\end{align}
and there exists at least one pair $(i,j)$ for which the sum in (\ref{sum1}) is positive.\\
We now recall an important result from \cite{Pach}

Let $M$ be an $m\textrm{-by-}n$ zero-one matrix . Consider the following properties of $M$.
\begin{enumerate}[(a)]
\item No collection of orthogonal cycles of $M$ is positive.
\item $M$ can be obtained from an $m\textrm{-by-}n$ real matrix $E = [e_{ij}]$ by replacing each $0$ entry of $E$ by $1$ and each nonzero entry by $0$.  Furthermore, for each $i$ and $j$ satisfying $1 \le i < m, 1 \le j < n$, the matrix $E$ satisfies $d_{i,j} := e_{i+1,j+1} - e_{i+1,j} - e_{i,j+1} + e_{i,j} > 0$.
\end{enumerate}
Then the following holds (\cite{Pach},part of Theorem 4):
\begin{theorem}\label{Pachzero}
For a zero-one $m\textrm{-by-}n$ matrix $M$, (a) $\iff$ (b).
\end{theorem}

For a zero-one matrix $M$, the \emph{weight} of $M$ is the number of $1$'s that appear as entries. From \cite[Theorem 3]{Pach}, we have the following asymptotic result.

\begin{theorem}\label{Pachone}
The maximum weight of an $n\textrm{-by-}n$ zero-one matrix containing no positive orthogonal cycle is $O(n^{4/3})$.
\end{theorem}

Finally, the multiplicity of an entry $a$ in the matrix $A$ is the number of occurrences of $a$ in $A$. For example, let $A=\left(
                                                                                                                             \begin{array}{ccc}
                                                                                                                                1& 2 & 3 \\
                                                                                                                                5& 2 & 4 \\
                                                                                                                                2& 7 & 2 \\
                                                                                                                             \end{array}
                                                                                                                           \right)$. Then the multiplicity of 2 in $A$ is 4, and the multiplicity of 8 in $A$ is 0. When we say: "the $k$-th smallest entry of a real matrix", we mean an entry for which there exist exactly $k-1$ different entries that are strictly smaller than this entry. For example, in the matrix $A$ above, the 1-st smallest entry is 1, the 2-nd smallest entry is 2, and the 3-rd smallest entry is 3.

\section{The totally nonsingular case}
In this section, we asymptotically bound the number of equal entries in a totally nonsingular $n\textrm{-by-}n$ matrix. We recall the following classical fact (which is an immediate consequence of, for example, Lemma 2.1 in \cite{Alon}).
\begin{lemma}\label{nalontwo} Let $P=P(x_1,\ldots,x_n)$ be a polynomial over $\mathbb{R}$.  Suppose there is some ball $B = B(c,\delta) \subset \mathbb{R}^n$ such that for all $x \in B$, $P(x) = 0$. Then $P \equiv 0$.\end{lemma}

Using this lemma, we prove the following:
\begin{lemma} \label{polys}
Let $\varepsilon >0$ and let $\lbrace p_1,\ldots,p_N\rbrace$ be a finite family of real non-zero polynomials in $k$ variables. Then there exists a ball $B = B(c,\delta) \subset B(0,\varepsilon) \subset \mathbb{R}^k$ (a ball with center in $c$ and radius $\delta$) such that for all $x \in B$ and $1 \leq i \leq N$ we have $p_i(x) \neq 0$.\end{lemma}

\begin{proof}
By continuity, it is enough to show that there exists $x \in B(0,\varepsilon)$ such that $p_i(x) \neq 0$ for all $1 \leq i \leq N$. Assume by contradiction that it is not true, and let $p=\prod_{i=1}^{N}p_i$. Then $p(x)=0$ for all $x \in B(0,\varepsilon)$, and hence from Lemma~\ref{nalontwo}  we get a contradiction.
\end{proof}

In \cite{Furedi} is the following result.

\begin{theorem}\label{Hajnal}
Let $M$ be an $n\textrm{-by-}n$ zero-one matrix with no $2\textrm{-by-}2$ submatrix consisting only of ones.  Then the maximal number of ones that $M$ can have is $\Theta(n^{3/2})$.
\end{theorem}

This theorem has important consequences for the number of equal entries in a totally nonsingular matrix. In order to prove it, we start with the following:

\begin{theorem}\label{tnscharacterization}
The following statements about an $n\textrm{-by-}n$ zero-one matrix $M$ are equivalent
\begin{enumerate}
\item[(f)] $M$ has no $2\textrm{-by-}2$ submatrix consisting only of ones.
\item[(g)] For any real number $b \neq 0$ there exists a totally nonsingular matrix $A$ such that $C_b(A) = M$.
\end{enumerate}
\end{theorem}

\begin{proof}
We start from $(f) \implies (g)$, and note that it is enough to prove the statement for $b=1$. Define $A = [a_{ij}]$ to be a matrix of variables
\begin{align*}
a_{ij} = \left\{\begin{array}{ll}
x_{ij}, &m_{ij} = 0 \\
1, & m_{ij} = 1
\end{array}\right. .
\end{align*}
The minors of $A$ determine a collection of nonzero polynomials, and by Lemma~\ref{polys}, one can choose variables $\big\{x_{ij}\big\}$ such that $A$ would be totally nonsingular matrix with $C_1(A) = M$.\\
For $(g) \implies (f)$, It is clear that if $M$ has a $2\textrm{-by-}2$ submatrix of ones, then the corresponding minor of $A$ would be zero, a contradiction.
\end{proof}

As a corollary to Theorems~\ref{Hajnal} and~\ref{tnscharacterization} we obtain the following:
\begin{corollary}
The maximal number of equal entries in a totally nonsingular $n\textrm{-by-}n$ matrix is $\Theta(n^{3/2})$.
\end{corollary}

\section{The totally positive case}
In this section, we asymptotically bound the number of equal entries in an $n\textrm{-by-}n$ $TP$ matrix.  This bound also applies to the number of equal minors of size $(n-1)\textrm{-by-}(n-1)$. Surprisingly, the bound that we get is noticeably smaller than the bound from the previous section.\\

Let $A \in M_{m,n}(\mathbb{R})$ be entry-wise nonnegative. Let $A^{(t)}$ denote the $t^{th}$ Hadamard power of $A$, that is $A^{(t)} = [a_{ij}^t]$ if $A = [a_{ij}]$.  We say that $A$ is \emph{eventually $TP$} if there exists an $N > 0$ such that for each $t > N$, the matrix $A^{(t)}$ is $TP$. Given in \cite{Fallat} is the following result:

\begin{theorem}\label{thm:TP2}
A matrix $A$ is $TP_2$ if and only if $A$ is eventually $TP$.
\end{theorem}

Thus, if $A$ is $TP_2$, there exists $t>0$ for which $A^{(t)}$ is $TP$, and for each entry $b$ of $A$, $C_b(A)=C_{b^t}(A^{(t)})$. As a result, the maximal number of equal entries in an $n\textrm{-by-}n$ $TP$ matrix is equal to the maximal number of equal entries in an $n\textrm{-by-}n $ $ TP_2 $ matrix. Our main result in this section is the following:

\begin{theorem}\label{equalTP}
The number of equal entries in an $n\textrm{-by-}n$ totally positive matrix is $O(n^{4/3})$, and for arbitrarily large values of $n$, there exist $n\textrm{-by-}n$ totally positive matrices with $\Omega(n^{4/3})$ equal entries.  Thus, the maximal number of equal entries in an $n\textrm{-by-}n$ totally positive matrix is $\Theta(n^{4/3})$.
\end{theorem}

Before presenting the proof, we recall an important theorem from \cite{Szemeredi}. A point $p = (x_0,y_0)$ is said to be \emph{incident} with a line $L = \{(x,y) : y = mx + b\}$ if $p \in L$.

\begin{theorem}\label{thm:szem}
Given $n$ points and $n$ lines in the plane, the maximal number of point-line incidences is $\Theta(n^{4/3})$.
\end{theorem}

We now prove the following lemma:
\begin{lemma} \label{lemmequalTP}
For arbitrarily large values of $n$, there exist $n\textrm{-by-}n$ totally positive matrices with $\Omega(n^{4/3})$ equal entries.
\end{lemma}

\begin{proof}
First, let us take a collection of $n$ points and $n$ lines for which the number of incidences is $\Theta(n^{4/3})$. According to \cite{Pach}, it is possible to apply a transformation on the set of points and lines that preserves the incidences, and such that after applying the transformation no two points would have the same $x$-coordinates, no two lines would have the same slope, and no line is parallel to the $y$-axis. Let $p_1=(x_1,y_1), \ldots, p_n=(x_n,y_n)$ be those $n$ points, such that $x_1 < x_2 < \ldots < x_n$, and let $l_1, \ldots, l_n$ be those $n$ lines (such that the slope $l_a$ is smaller than the slope of $l_b$ for $a < b$). Denote by $a_{ij}$ be the vertical distance from point $p_i$ to line $l_j$. The vertical distance is defined as follows: Let us draw a line that is parallel to the $y$-axis and passes through $p_i$, and let $d_{ij}$ be the length of the segment of the line whose one end is $p_i$ and whose other end is the point in which this line intersects $l_j$. If this point of intersection is above $p_i$ then $a_{ij}=d_{ij}$. Otherwise, $a_{ij}=-d_{ij}$. Note that if $p_i$ incidents with $l_j$ then $a_{ij}=0$. Let us form an $n\textrm{-by-}n$ matrix $A = [a_{ij}]$. We will show that for $i,k,u,j$ such that $1 \leq i < k \leq n$, $1 \leq u < j \leq n$, $d = a_{iu}+a_{kj}-a_{ij}-a_{ku} > 0$. Define $d_1 = a_{iu} - a_{ij} $ and $d_2 = a_{kj} - a_{ku}$.  The lines $l_j$ and $l_u$ are not parallel and hence cross at exactly one point, say $q = (x,y)$. Without loss of generality, we may assume that $a_{ij} > 0, a_{iu} > 0, a_{kj} > 0$ and $a_{ku} > 0$ (otherwise we can "push down" the points $p_i=(x_i,y_i)$ and $p_k=(x_k,y_k)$ by the same distance, without changing the value of $d$.) There are three cases to consider:

\emph{Case 1.} $x_i < x_k \leq x$. Then $d_1 > 0$ and $d_2 \leq 0$. However, $|d_1| > |d_2|$ and thus $d= d_1 + d_2 = |d_1| - |d_2| > |d_1| - |d_1| = 0$.

\emph{Case 2.} $x_i < x < x_k$.  Then $d_1 > 0$ and $d_2 > 0$.  Hence $d= d_1 + d_2 > 0$.

\emph{Case 3.} $x \leq x_i < x_k$. Then $d_1 \leq 0$ and $d_2 > 0$.  However, $|d_2| > |d_1|$ and thus $d = d_1 + d_2 = -|d_1| + |d_2| > -|d_2| + |d_2| =0 $.

Consider the $n\textrm{-by-}n$ matrix $B$ for which $b_{ij}=t^{a_{ij}}$, where $t$ is some fixed real number such that $t > 1$. Note that since $d>0$, we may conclude that $B$ is a $TP_2$ matrix and the entries that are equal to 1 in $B$ correspond to the point-line incidences (thus both of those quantities are equal).  By taking Hadamard power of $B$ as in Theorem \ref{thm:TP2}, we can obtain a $TP$ matrix for which the number of entries that are equal to one is the same as the number of point-line incidences.

\end{proof}

We are now ready to present a proof of Theorem~\ref{equalTP}:
\begin{proof}
Let $A$ be an $n\textrm{-by-}n$ $TP$ matrix, and without loss of generality, assume that 1 is the entry with the highest multiplicity in $A$. Define an $n\textrm{-by-}n$ matrix $E$ for which $e_{ij}=\log_{2}(a_{ij})$. Since $A$ is $TP$ the matrix $E$ satisfies condition (b) (before Theorem~\ref{Pachone}). Therefore, from Theorems~\ref{Pachzero} and~\ref{Pachone}, we conclude the first part of Theorem~\ref{equalTP}. The second part is due to Lemma~\ref{lemmequalTP}, and we are done.
\end{proof}

This asymptotic bound on the number of equal entries also applies to the number of equal $(n-1)\textrm{-by-}(n-1)$ minors, using a fact from \cite{Fallat} or \cite{Gantmacher}.

If $A$ is an $m\textrm{-by-}n$ matrix over a field and $k \le m, n$, then the ${{m}\choose{k}}\textrm{-by-}{{n}\choose{k}}$ matrix of $k\textrm{-by-}k$ minors of $A$ (with index sets ordered lexicographically) is the \emph{$k^{th}$ compound of $A$}.

\begin{corollary}
The maximal number of equal $(n-1)\textrm{-by-}(n-1)$ minors in an $n\textrm{-by-}n$ $TP$ matrix is $\Theta(n^{4/3})$.
\end{corollary}

\begin{proof}
Let $A$ be an $n\textrm{-by-}n$ $TP$ matrix.  By \cite[p. 75]{Gantmacher}, the $(n-1)^{st}$ compound of $A$ is totally positive.  Theorem \ref{equalTP} then implies that the maximal number of equal $(n-1)\textrm{-by-}(n-1)$ minors is therefore $\Theta(n^{4/3})$.
\end{proof}

In this section we discussed the number of equal entries and equal $(n-1)\textrm{-by-}(n-1)$ minors in $TP$ matrices. A natural generalization of this would be the following: What is the maximal possible number of equal $k\textrm{-by-}k$ minors in an $m\textrm{-by-}n$ $TP$ matrix? Note that for for the $n\textrm{-by-}n$ case, for $2 \leq k \leq n-2$, the $k^{th}$ compound of $TP$ matrix is not necessarily $TP$, so we cannot apply Theorem \ref{equalTP} in order to give an answer to this question. We will discuss some aspects of this question in the last section.

\section{The multiplicity of the smallest (largest) k entries in a $TP$ matrix}
In the previous section we calculated the maximal multiplicity that an entry in a totally positive matrix may have. However, we didn't discuss the size of this entry, relatively to other entries in the matrix. For example, is it possible that this entry is the smallest or the largest entry in the matrix? For $n$ large enough, the answer is no. Moreover, we prove the following:
\begin{theorem}\label{thm:kthsmallestentries}
Let $A$ be an $n\textrm{-by-}n $ $ TP$ matrix. Then the total multiplicity of the smallest (or largest) $k$ entries is at most $(2^k-1)(2n-1)$.
\end{theorem}
\begin{proof}
We start the proof by considering the smallest $k$ entries. Consider the set of \emph{diagonals} of $A$: $d_c = \{a_{ij} \in A: i= c + j\}$, $1-n \leq c \leq n-1$. We show that a diagonal of $A$ can contain at most $2^k-1$ entries that are among the $k$ smallest entries in $A$. Note that once we show that, we are done since there are $2n -1$ diagonals. We prove it by induction on $k$. For $k=1$, no two smallest entries may lie on the same diagonal (otherwise, those two entries are the main diagonal of some $2\textrm{-by-}2$ submatrix of $A$ that has nonpositive determinant, which is a contradiction). Thus, each diagonal of $A$ contains at most $1=2^1-1$ entry that is the smallest entry in the matrix, and hence for $k=1$ we are done. Suppose $k>1$, and assume in contradiction that $A$ has a diagonal $d_c$ with $2^k$ such entries (if there are more, we just choose a set of $2^k$ among them). We denote the indices of those entries by $\{i_v+c,i_v\}_{v=1}^{2^k}$, and define a $2^k\textrm{-by-}2^k$ matrix $B$ to be $B=A[i_1+c,i_2+c, \ldots,i_{2^k}+c|i_1,i_2, \ldots,i_{2^k}]$. Consider the super diagonal of $B$ (its entries appear on the diagonal $d_{c-1}$ in $A$). It has $2^k-1$ entries, and by the inductive hypothesis, in cannot contain more than $2^{(k-1)}-1$ entries that are strictly smaller than the $k$-th smallest entry in $A$. On the other hand, it also cannot have more than  $2^{(k-1)}-1$  entries that are equal to or greater than the $k$-th smallest entry in $A$ (otherwise, since $B$ is $TP_2$, its subdiagonal must have more than $2^{(k-1)}-1$ entries that are strictly smaller than the $k$-th smallest entry in $A$, which is a contradiction to the inductive assumption). So in conclusion, the supper diagonal of $B$ has at most $2^{(k-1)}-1$ entries that are strictly smaller than the $k$-th smallest entry in $A$, and at most $2^{(k-1)}-1$  entries that are equal to or greater than the $k$-th smallest entry in $A$. Hence, it has at most $2( 2^{(k-1)}-1 )=2^k-2$ entries, and we get a contradiction. Therefore the multiplicity of the $k$ smallest entries in $A$ is at most $(2^k-1)(2n-1)$. For the multiplicity of the $k$ largest entries, consider the set of anti-diagonals of $A$, and the result follows similarly.

\end{proof}

Combining this theorem with Theorem~\ref{equalTP} leads to an interesting conclusion. Let $A$ be an $n\textrm{-by-}n$ $TP$ matrix that has an entry with multiplicity $\Theta(n^{4/3})$. We are interested in the relative size of this entry with respect to other entries in $A$. Let us assume that this is the $k$-th smallest entry. Then using Theorems~\ref{thm:kthsmallestentries} and~\ref{equalTP}, for some constant $c$ we have $cn^{4/3} \leq (2^k-1)(2n-1)$. Therefore $cn^{4/3} \leq 2^{k+1}n$, and hence $cn^{1/3} \leq 2^{k+1}$. Thus, for $n$ big enough, we have ${1/3}\log n\leq k $. Similarly, if this entry is the $r$-th largest entry, we must have ${1/3}\log n\leq r $. So, apparently, this "high multiplicity entry" cannot be too small or too large relative to the other entries in the matrix.\\

Note that this observation also has a geometrical interpretation. Suppose we have a collection of $n$ points and $n$ lines in the plane with a high number of incidences (such as $\Theta(n^{k})$ for $k>1$), such that no two points have the same $x$-coordinate, no two lines have the same slope, and no line is parallel to the $y$-axis. Then for $n$ large enough, using the connection between $TP$ matrices and point-line incidences that we discussed in the previous section, we can conclude that there are at least $(k-1)\log n$ pairs of points and lines for which the point is below the line, and at least the same number of such pairs for which the point is above the line.
\section{Applications to matrix completion problems}

Our results have important implications for $TP$ matrix completions. We exhibit the incompletability of certain patterns of data that involve positive orthogonal cycles.

We say that a rectangular array is a \emph{partial matrix} if some of its entries are specified while others are left to be chosen.  The specified entries are called \emph{data}.  We say that a partial matrix is \emph{partial $TP$} if each minor consisting of specified data is positive.  A $TP-$ \emph{completion} of a partial $TP$ matrix is a choice of the unspecified entries such that the resulting matrix is $TP$.

A \emph{pattern} is an equivalence class of partial matrices for which two partial matrices are the same pattern if the unspecified entries are in the same positions and if the data are in the same positions.  A pattern is \emph{$TP$ completable} if for any choice of data that gives a partial $TP$ matrix, there exists a completion of the partial matrix that is $TP$.  Otherwise, we say that a pattern is not $TP$ completable.

We now apply our work on positive orthogonal cycles.  Let $A$ be a pattern.  Let $C = (p_0, \ldots, p_{2k})$ be an orthogonal cycle.  We say that $C$ is an \emph{orthogonal cycle of $A$} if each position $p_i$ is to be a specified entry in $A$. (Note that this notation is slightly different from that which was previously introduced in that we have changed the requirement that $p_i = 1$.)

\begin{theorem}\label{thm:comp}
Suppose that A is a pattern that has no $2\textrm{-by-}2$ submatrix of specified data. If  some collection of orthogonal cycles of $A$ is positive, then $A$ is not $TP_2$ completable.
\end{theorem}

\begin{proof}
Choose the specified entries of $A$ to be $1$'s.  We may do this since $1>0$ and since the $1$'s form no $2\textrm{-by-}2$ submatrix. Hence, the resulting matrix is partial $TP_2$. Suppose that $A = [a_{ij}]$ is a $TP_2$ completion.  Choose $x > 0$ and let $e_{ij} = \log(a_{ij})/\log(x)$.  Then $E$ satisfies property (b) (before Theorem~\ref{Pachone}). However, the corresponding $M$ has a collection of positive orthogonal cycles, which is a contradiction.
\end{proof}

\begin{corollary}
If $A$ is as in Theorem \ref{thm:comp}, then $A$ is, in particular, not $TP$ completable.
\end{corollary}

As an example, consider the pattern $\left(
  \begin{array}{cccc}
    ? & x & ? & x \\
    ? & ? & x & x \\
    x & x & ? & ? \\
    x & ? & x & ? \\
  \end{array}
\right)$. We saw in section 2 that the ones in the matrix $\left(
  \begin{array}{cccc}
    0 & 1 & 0 & 1 \\
    0 & 0 & 1 & 1 \\
    1 & 1 & 0 & 0 \\
    1 & 0 & 1 & 0 \\
  \end{array}
\right)$ form a positive orthogonal cycle, and thus by Theorem \ref{thm:comp} this pattern is not $TP_2$ (or $TP$) completable. More about $TP_2$ completablity can be found in \cite{Shahla}

\section{Relation to Bruhat order on permutations}
The Bruhat order on $S_n$ (as defined in, for example \cite{Drake}) compares two permutations $\pi$, $\sigma$ according to the following criterion: $\pi \leq \sigma$ if $\sigma$ is obtainable from $\pi$ by a sequence of transpositions $(i,j)$ where $i < j$ and $i$ appears to the left of $j$ in $\pi$. For example, one can obtain $1324 \in S_4$ from $3412 \in S_4$ through the following sequence of transpositions:

\begin{center}
$3412 \Rightarrow 3142 \Rightarrow 3124 \Rightarrow 1324,$
\end{center}
and hence $1324 \leq 3412$.

There is an analogue way to compare permutations that is useful here. Let $M(\pi)$ be the matrix whose $ij$ entry is 1 if $j = \pi(i)$ and zero otherwise. Defining $[i] = \{1, . . . , i\}$, and denoting the submatrix of $M(\pi)$ corresponding to rows $I$ and columns $J$ by $M(\pi)_{I,J}$, we have the following (\cite{Drake}):
\begin{theorem}\label{thm:bruhat}
Let $\pi, \sigma \in S_n$. Then $\pi \leq \sigma$ in the Bruhat order if and only if for all $1 \leq i, j \leq n$, the number of ones in $M(\pi)_{[i],[j]}$ is greater than or equal to the number of ones in $M(\sigma)_{[i],[j]}$.
\end{theorem}

Let $A$ be an $n\textrm{-by-}n$ zero-one matrix that has an orthogonal cycle $C = (p_0,p_1,\ldots, p_{2n})$ for which every row and every column include exactly two elements from the sequence, and such that all the entries that do not correspond to $C$ are zero. Note that from the definition of $C$ and $A$, one can write $A$ as a sum of two permutation matrices, such that the first one corresponds to positions $(p_0,p_2,p_4,\ldots, p_{2n-2})$ (the even positions) of $C$, and the second one corresponds to positions $(p_1,p_3,p_5,\ldots, p_{2n-1})$ (the odd positions) of $C$ (we assume that $p_0$ is to the left of $p_1$ and they are both in the first row. Other cases will be treated later). Let $\pi$ be the first permutation and $\sigma$ be the second permutation. Then, $C$ is positive orthogonal if and only if the number of ones in $M(\pi)_{[i]^c,[j]^c}$ is greater than or equal to the number of ones in $M(\sigma)_{[i]^c,[j]^c}$ for all $(i,j)$. Note that in the original definition, it is required that there exists a pair $(i,j)$ for which instead of "greater than or equal to" one could write "greater than" in the sentence above. However, this additional requirement is satisfied automatically. This is due to the following reason: suppose that the position of $p_0$ is $(1,u)$ (and we must have $u<n$). Consider $M(\pi)_{[1]^c,[u]^c}$. it must have exactly $n-u$ ones (since $\pi$ is a permutation). However, $M(\sigma)_{[1]^c,[u]^c}$ has $n-u-1$ ones (since $p_1$ is to the right of $p_0$). Thus, for the pair $(i,j)=(1,u)$ we get strict inequality. Finally, using almost the same proof as for Theorem~\ref{thm:bruhat}, we conclude the following:
\begin{theorem}
Let $C$, $\pi$ and $\sigma$ be as defined above. Then $C$ is a positive orthogonal cycle if and only if $\pi \leq \sigma$ in Bruhat order.
\end{theorem}
In this discussion we assumed that $p_0$ is to the left of $p_1$ and they are both in the first row. We now generalize this discussion to orthogonal cycles that are not sums of two permutations, as well as collections of orthogonal cycles. Let $M$ be a
(0,1) $m\textrm{-by-}n$ matrix and let $\mathcal{O} = \{C_a\}_{a \in A}$ be an orthogonal collection that corresponds to $M$. Let $M_1$ and $M_2$ be two (0,1)-matrices, $M_1$ corresponds to the odd positionings of cycles in $\mathcal{O}$, and $M_2$
corresponds to the even positionings. Then by definition, $\mathcal{O}$ is positive orthogonal collection if and only if
the number of ones in ${M_2}_{[i]^c,[j]^c}$ is greater than or equal to the number of ones
in ${M_1}_{[i]^c,[j]^c}$ for all $(i,j)$ (it is easy to show that there exists a pair $(i,j)$ for which it is strictly bigger).
We now introduce the notion of Bruhat order on subclasses of $(0,1)$- matrices, as it appears in \cite{Deaett}
Let $R = (r_1, r_2,\ldots,r_m)$ and $S = (s_1, s_2,\ldots,s_n)$ be positive integral vectors. Then $A(R, S)$
denotes the class of all $m\textrm{-by-}n$ $(0,1)$-matrices with row sum vector $R$ and column sum vector $S$.
If $A_1$ and $A_2$ are in $A(R, S)$, then $A_1 \leq A_2$ in the Bruhat order on $A(R, S)$ if and only if for all $1 \leq i \leq m$,$1 \leq j \leq n$, the number of ones
in ${A_1}_{[i]^c,[j]^c}$ is greater than or equal to the number of ones in ${A_2}_{[i]^c,[j]^c}$. We now conclude the following:
\begin{theorem}
Let $\mathcal{O}$, $M_1$ and $M_2$ as be as defined above. Then $\mathcal{O}$ is a positive orthogonal collection if and only if $M_2 \leq M_1$ in Bruhat
order.
\end{theorem}

We conclude with the following observation. Let $A$ be a totally positive matrix, and let $b$ be some entry in $A$. Consider the matrix $E$ for which $E_{ij}=\log(\frac{a_{ij}}{b})$. Note that $E_{ij}=0$ if and only if $A_{ij}=b$, and that since $A$ is $TP$, the matrix $E$ satisfies condition $(b)$ (before Theorem~\ref{Pachone}). Thus, according to Theorem~\ref{Pachzero}, for any entry $b$ in $A$, no collection of orthogonal cycles of $C_b(A)$ is positive.

\section{Equal $2\textrm{-by-}2$ minors}
In the previous sections we discussed the number and the positioning of equal entries in $TP$ matrices, as well as equal $(n-1)\textrm{-by-}(n-1)$ minors. It is natural to ask what can be said about the number and positioning of equal $k\textrm{-by-}k$ minors for some $2 \leq k \leq n-2$. Here we discuss the case $k=2$, and concentrate on $2\textrm{-by-}n$ $TP$ matrices. We present surprising relations between the positioning of equal $2\textrm{-by-}2$ minors and outer planar graphs.
\\
We start this section with several examples. Consider the following matrices:

$\left(
   \begin{array}{ccc}
     1 & 2 & 1 \\
     6 & 18 & 12 \\
   \end{array}
 \right),  \left(
             \begin{array}{cccc}
               1 & 2 & 3 & 1 \\
               6 & 18 & 30 & 12 \\
             \end{array}
           \right), \left(
                      \begin{array}{ccccc}
                        1 & 2 & 3 & 4 & 1 \\
                        6 & 18 & 30 & 42 & 12 \\
                      \end{array}
                    \right)
$. The number of equal $2\textrm{-by-}2$ minors in each one of those matrices is 3, 5 and 7 respectively, and from the next lemma we also get that those are the maximal possible numbers of equal $2\textrm{-by-}2$ minors. For a $2\textrm{-by-}n$ $TP$ matrix $A$, define $\alpha_A = \big\{(i,j) | det(A[{1,2}|{i,j}])=\alpha \big\}$.
\begin{lemma}\label{tppossible}
Let $A=[a_{i,j}]$ be a $2\textrm{-by-}n$ $TP$ matrix, and let $1 \leq i<j<k<w \leq n$. Then for any positive $\alpha$:\\
\begin{itemize}
  \item $\big\{ (i,j), (i,k), (j,w), (k,w) \big\} \not\subseteq \alpha_A$
  \item $\big\{ (i,k), (i,w), (j,k), (j,w) \big\} \not\subseteq \alpha_A$
\end{itemize}
\end{lemma}
\begin{proof}
Consider the first conclusion. Assume in contradiction that\\ $\big\{ (i,j), (i,k), (j,w), (k,w) \big\}\subseteq \alpha_A$, and hence
\begin{center}
$\det\left(
   \begin{array}{cc}
     a_{1,i} & a_{1,j} \\
     a_{2,i} & a_{2,j} \\
   \end{array}
 \right)=
 \det\left(
\begin{array}{cc}
     a_{1,j} & a_{1,w} \\
     a_{2,j} & a_{2,w} \\
   \end{array}
 \right)=
 \det\left(
   \begin{array}{cc}
     a_{1,i} & a_{1,k} \\
     a_{2,i} & a_{2,k} \\
   \end{array}
 \right)=
 \det\left(
   \begin{array}{cc}
     a_{1,k} & a_{1,w} \\
     a_{2,k} & a_{2,w} \\
   \end{array}
 \right) = \alpha$.
\end{center}

Therefore,\\ $0=\det\left(
                 \begin{array}{cc}
                   a_{1,i} & a_{1,j} \\
                   a_{2,i} & a_{2,j} \\
                 \end{array}
               \right)-\det\left(
              \begin{array}{cc}
                   a_{1,j} & a_{1,w} \\
                   a_{2,j} & a_{2,w} \\
                 \end{array}
               \right)=\\=\det\left(
                            \begin{array}{cc}
                              a_{1,i} & a_{1,j} \\
                              a_{2,i} & a_{2,j} \\
                            \end{array}
                          \right)+\det\left(
                         \begin{array}{cc}
                              a_{1,w} & a_{1,j} \\
                              a_{2,w} & a_{2,j} \\
                            \end{array}
                          \right)=\det\left(
                                          \begin{array}{cc}
                                            a_{1,i}+a_{1,w} & a_{1,j} \\
                                            a_{2,i}+a_{2,w} & a_{2,j} \\
                                          \end{array}
                                        \right)$,\\
                                        and,\\ $0=\det\left(
                                                  \begin{array}{cc}
                                                    a_{1,i} & a_{1,k} \\
                                                    a_{2,i} & a_{2,k} \\
                                                  \end{array}
                                                \right)-\det\left(
                                               \begin{array}{cc}
                                                    a_{1,k} & a_{1,w} \\
                                                    a_{2,k} & a_{2,w} \\
                                                  \end{array}
                                                \right)=\\=\det\left(
                                                             \begin{array}{cc}
                                                               a_{1,i} & a_{1,k} \\
                                                               a_{2,i} & a_{2,k} \\
                                                             \end{array}
                                                           \right)+\det\left(
                                                          \begin{array}{cc}
                                                               a_{1,w} & a_{1,k} \\
                                                               a_{2,w} & a_{2,k} \\
                                                             \end{array}
                                                           \right)=\det\left(
                                                                           \begin{array}{cc}
                                                                             a_{1,i}+a_{1,w} & a_{1,k} \\
                                                                             a_{2,i}+a_{2,w} & a_{2,k} \\
                                                                           \end{array}
                                                                         \right)$.\\

Since $A$ is $TP$, all its entries are positive, and hence we get that the $j^{th}$ and $k^{th}$ columns of $A$ are scalar multiples of $\left(
                                                                         \begin{array}{c}
                                                                           a_{1,i}+a_{1,w}\\
                                                                           a_{2,i}+a_{2,w}\\
                                                                         \end{array}
                                                                       \right)$. Thus, $\det A[1,2|j,k]=0$, in contradiction to the fact that $A$ is $TP$. For
the second conclusion, assume for contradiction that $\big\{ (i,k), (i,w), (j,k), (j,w) \big\} \subseteq \alpha_A$. Since $A$ is $TP$, $\left(
                                                                                                                                    \begin{array}{c}
                                                                                                                                      a_{1,i} \\
                                                                                                                                      a_{2,i} \\
                                                                                                                                    \end{array}
                                                                                                                                  \right) \neq \left(
                                                                                                                                                 \begin{array}{c}
                                                                                                                                                   a_{1,j} \\
                                                                                                                                                   a_{2,j} \\
                                                                                                                                                 \end{array}\right)$, and in a
similar way, we get that $\det A[1,2|k,w]=0$, which is again a contradiction.
\end{proof}
In order to understand more deeply the structure of equal $2\textrm{-by-}2$ minors, we associate graphs with $2\textrm{-by-}n$ $TP$ matrices in the following way:
Let $G$ be a graph on $n$ vertices. We say that $G$ is \emph{$TP$ attainable} if there exists a labeling of its vertices, a $2\textrm{-by-}n$ $TP$ matrix $A$ and a positive number $\alpha$ such that $E(G)\subseteq \alpha_A$ (note that from this definition, if $G$ is $TP$ attainable, then any graph that is obtained from $G$ by removing an edge is also $TP$ attainable). We call such a matrix $A$ a realization of $G$. Using the example from the beginning of the section, we get that $C_4$ is attainable, as well as the graph in figure 1
\begin{figure}[h!]
\caption{}
\centering
\includegraphics[width=0.3\textwidth]{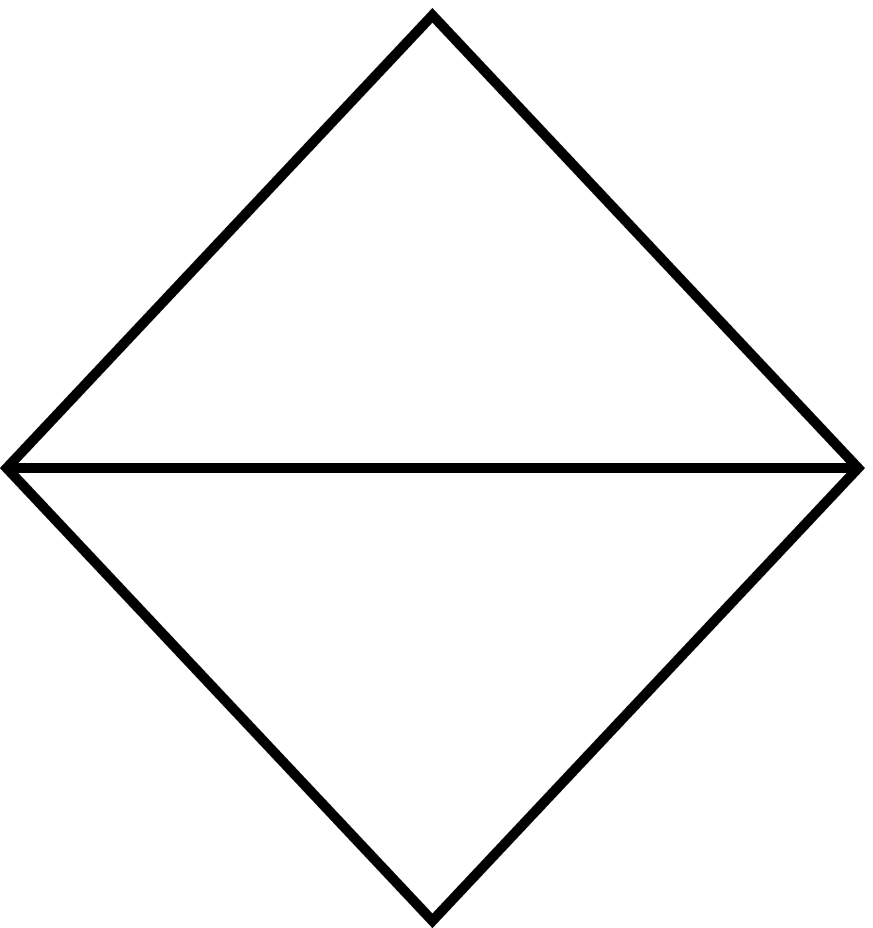}
\vspace{-0.85\baselineskip}
\label{first-fig}
\end{figure}
\\
,using the matrix $\left(
                   \begin{array}{cccc}
                     1 & 2 & 3 & 1 \\
                     6 & 18 & 30 & 12 \\
                   \end{array}
                 \right)$. From lemma~\ref{tppossible}, it is clear than not any labeling works, but there exists an appropriate labeling, as presented in figure 2.
\begin{figure}[h!]
\caption{}
\centering
\includegraphics[width=0.3\textwidth]{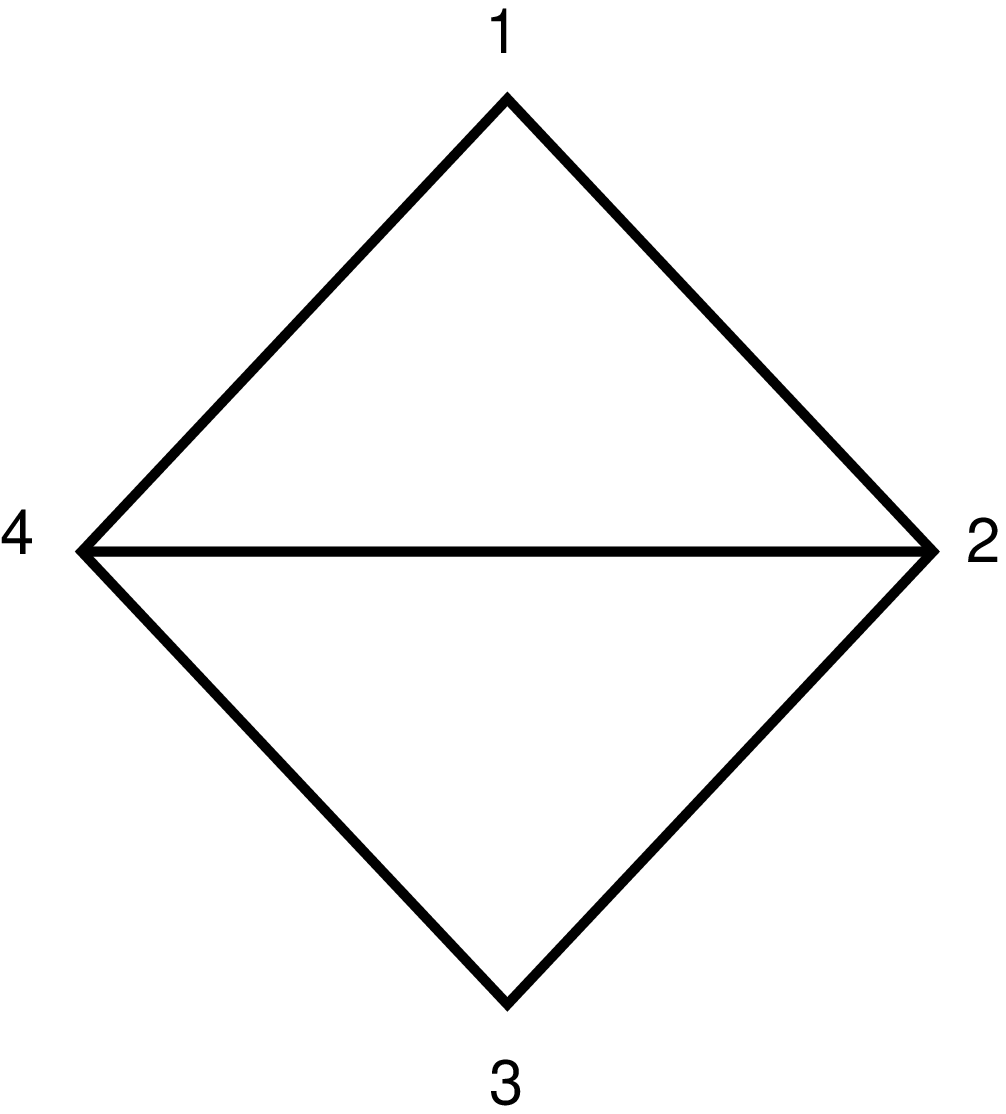}
\vspace{-0.85\baselineskip}
\label{first-fig}
\end{figure}
\\ In order to understand what are the possible positionings of equal $2\textrm{-by-}2$ minors, as well as the maximal number of such minors, we would like to have a full characterization of the set of $TP$ attainable graphs. From Lemma~\ref{tppossible}, it is clear that not every graph is a $TP$ attainable graph. For example, the lemma implies that $K_4$ and $K_{2,3}$ are not $TP$ attainable graphs. We are now ready to present a wide family of $TP$ attainable graphs- the outer planar graphs. A graph is called \emph{outerplanar} if it can be drawn in the plane without crossings in such a way that all of the vertices belong to the unbounded face of the depiction. That is, no vertex is totally surrounded by edges.
\begin{theorem}\label{planar}
Let $G$ be an outerplanar graph. Then $G$ is $TP$ attainable graph.
\end{theorem}
\begin{proof}
First, note that it is enough to prove  this theorem for maximal outer planar graphs (in such graphs, every bounded face is a triangle). Let $G$ be a maximal outerplanar graph. We call a labeling of the vertices in $G$ \emph{clockwise} if this labeling is obtained by choosing some vertex, labeling it with "1", and then walking clockwise on the unbounded face of $G$ and labeling the vertices with "2", "3", and so on. The labeling in figure 2 is an example of a clockwise labeling. We prove by induction that every maximal outerplanar graph is $TP$ attainable using a clockwise labeling. For $n \in \{2,3\}$ it is clear using the $2\textrm{-by-}3$ matrix from the beginning of the section. Now, assume that the statement is true for all the numbers smaller than $n$, and let $G$ be a maximal outer
planar graph on $n$ vertices. $G$ has a vertex $v$ of degree 2, and we denote its neighbors by $u$ and $w$. Consider the graph $G \setminus v$. By the inductive assumption, this graph is $TP$ attainable with a clockwise labeling, and let $y$ be the vertex with the label "1" in $G \setminus v$ (it is possible to have $y \in \{u,w\}$). Without loss of generality, there are two cases. a) The labeling of $u$  and $w$ in $G \setminus v$ is $i$ and $i+1$ respectively (for $1 \leq i \leq n-2$).
b) The labeling of $u$ and $w$ in $G \setminus v$ is $n-1$ and $1$ respectively. We examine each one of the cases, starting from a). Let $A_v$ be a realization of $G \setminus v$ that corresponds to the clockwise labeling ($A_v$ is of order $2\textrm{-by-}(n-1)$). Let $A$ be a $2\textrm{-by-}n$ matrix in which
$A[1,2|1,2,\ldots,i]= A_v[1,2|1,2,\ldots,i] $, $A[1,2|i+2,i+3,\ldots,n]= A_v[1,2|i+1,i+2,\ldots,n-1] $ and $A[1,2|i+1]= A_v[1,2|i]+ A_v[1,2|i+1]$.  We will show now that $A$ is a realization of $G$ that corresponds to the clockwise labeling for which the label of $y$ is "1". First, note that $\det A[1,2|i,i+1]=
\det \left(
                     \begin{array}{cc}
                       a_{1,i} & a_{1,i+1} \\
                       a_{2,i} & a_{2,i+1} \\
                     \end{array}
                   \right)= \\=\det \left(
                                                 \begin{array}{cc}
                                                   a_{v_{1,i}} & a_{v_{1,i}}+a_{v_{1,i+1}} \\
                                                   a_{v_{2,i}} & a_{v_{2,i}}+a_{v_{2,i+1}} \\
                                                 \end{array}
                                               \right)=\det \left(
                                                                              \begin{array}{cc}
                                                                                a_{v_{1,i}} & a_{v_{1,i+1}} \\
                                                                                a_{v_{2,i}} & a_{v_{2,i+1}} \\
                                                                              \end{array}
                                                                           \right)=  \alpha,  $\\ and similarly, $\det A[1,2|i+1,i+2]= \alpha$. In addition,
from Corollary 3.1.6 in \cite{Fallat}, a $2\textrm{-by-}n$ matrix is $TP$ if and only if all its contiguous $2\textrm{-by-}2 $ submatrices are $TP$. Therefore we get that $A$ is $TP$, and case a) is proved. Now let us consider case b). Let $A_v$ be a realization of $G \setminus v$ that corresponds to the clockwise labeling, and let $p$ and $q$ be the first and the last column
in $A_v$ respectively. Since $A_v$ is $TP$, $p$ and $q$ are independent, and for each $1 \leq i \leq n-1$ there exists a unique pair of real numbers $k_i,l_i$ such that the $i^{th}$ column of $A_v$ equals $k_ip+l_iq$ (note that $k_1=1, l_1=0, k_{n-1}=0, l_{n-1}=1$). We now obtain another realization of $G \setminus v$ (using the same labeling), denoted by $B_v$. We define the first and the last column of
$B_v$ to be $2p$ and $0.5q$ respectively. The $i^{th}$ column of $B_v$ would be $k_i(2p)+l_i(0.5q)$. Note that for any $1 \leq i <j \leq n-1$,
$\det B_v[1,2|i,j]=(k_il_j-l_ik_j)\det\left(
                                        \begin{array}{cc}
                                          p & q \\
                                        \end{array}
                                      \right)=\det A_v[1,2|i,j]$, and hence $B_v$ is a realization of $G \setminus v$ with the same labeling. We can repeat this process
until we obtain a realization $B_v$ of $G \setminus v$ for which the first column (denoted by $p$) is entrywise bigger than the last column (denoted by $q$). Consider a $2\textrm{-by-}n$ matrix $A$ such that $A[1,2|1]=p-q$, $A[1,2|2,3,\ldots,n]= B_v$. It is easy to see that $A$ is a realization of $G$ with the clockwise labeling such that $v$ is
labeled with "1", and thus we finished the proof of case b).
\end{proof}
Using this theorem, we obtain the following corollary
\begin{corollary}
The maximal number of equal $2\textrm{-by-}2$ minors in a $2\textrm{-by-}n$ $TP$ matrix is at least $2n-3$.
\end{corollary}
Theorem~\ref{planar} states that the set of outerplanar graphs is a subset of $TP$ attainable graphs. The converse, however, is not true. The graph in figure 3 is not outerplanar, and yet it is a $TP$ attainable graph, as is shown by the matrix $\left(
                                                                                                                        \begin{array}{cccccc}
                                                                                                                          8 & 34 & 9 & 14 & 20 & 6 \\
                                                                                                                          4 & 24 & 8 & 14 & 24 & 10 \\
                                                                                                                        \end{array}
                                                                                                                      \right)
$.
\begin{figure}[h!]
\caption{}
\centering
\includegraphics[width=0.3\textwidth]{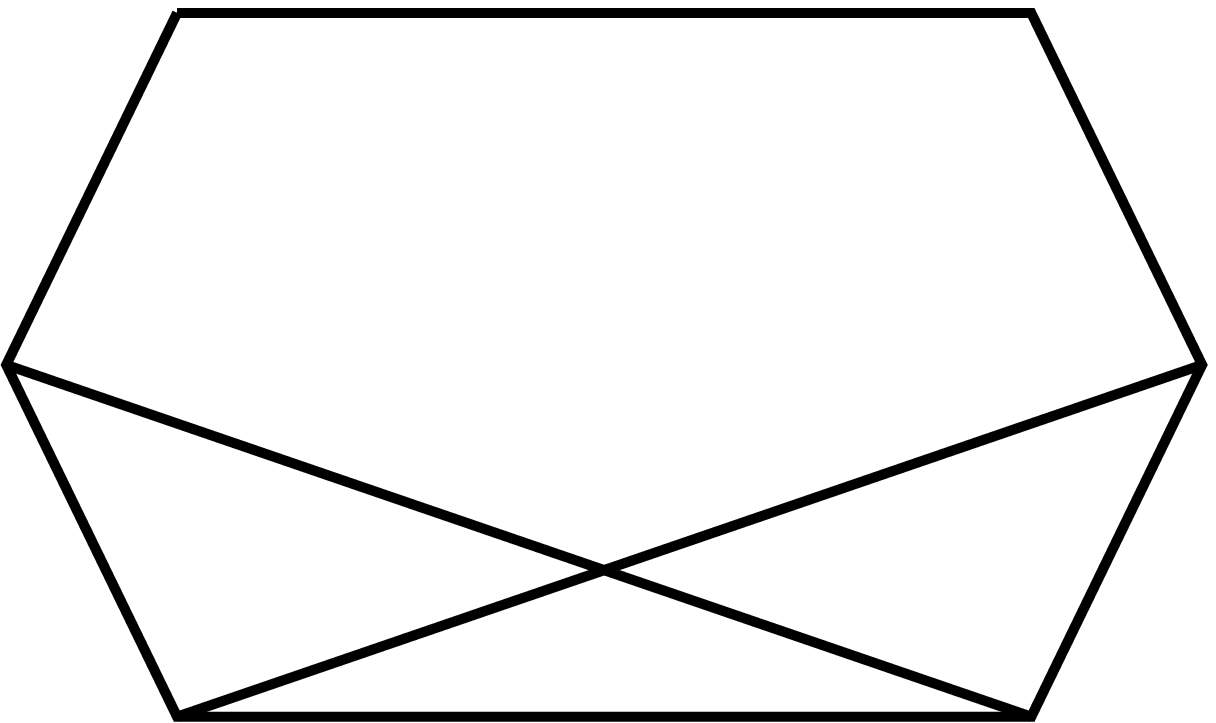}
\vspace{-0.85\baselineskip}
\label{first-fig}
\end{figure}
\\                                                                                                                                                                                                    \end{document}